\documentclass{amsart}

\newtheorem{theorem}{Theorem}[section]
\newtheorem{lemma}[theorem]{Lemma}
\newtheorem*{conjecture}{Conjecture}
\newtheorem*{chen conjecture}{Chen's Conjecture}
\newtheorem*{generalized chen conjecture}{Generalized Chen's Conjecture}

\theoremstyle{definition}

\newtheorem{proposition}[theorem]{Proposition}

\theoremstyle{remark}

\numberwithin{equation}{section}

\begin{document}

\title[Four dimensional biharmonic hypersurfaces]{Four dimensional biharmonic hypersurfaces in nonzero space form have constant mean curvature}

\author{Zhida Guan}
\address{Department of Mathematical Sciences, Tsinghua University, Beijing 100084, People's Republic of China}
\email{gzd15@mails.tsinghua.edu.cn}

\author{Haizhong Li}
\address{Department of Mathematical Sciences, Tsinghua University, Beijing 100084, People's Republic of China}
\email{lihz@tsinghua.edu.cn}

\author{Luc Vrancken}
\address{LMI-Laboratoire de Math\'{e}matiques pour l'Ing\'{e}nieur, Universit\'{e} Polytechnique Hauts de France, 59313 Valenciennes, France,  and KU Leuven, Department of Mathematics, Leuven, Belgium}
\email{luc.vrancken@uphf.fr}


\begin{abstract}
In this paper, through making careful analysis of Gauss and Codazzi equations, we prove that four dimensional biharmonic hypersurfaces in nonzero space form have constant mean curvature. Our result gives the positive answer to the conjecture proposed by Balmus-Montaldo-Oniciuc in 2008 for four dimensional hypersurfaces.
\end{abstract}

\maketitle

\section{Introduction}

Biharmonic maps were introduced in 1964 by Eells and Sampson [ES] as a generalization of harmonic maps. In their paper, Eells and Sampson suggested considering the bi-energy of a map $\phi:(M^n,g)\rightarrow (N^m,h)$ between two Riemannian manifolds defined by
\begin{equation}
E_2(\phi)=\frac{1}{2}\int_{M^n} |\tau(\phi)|^2 d\mu_g,
\end{equation}
where $\tau(\phi)$ is the tension field of $\phi$ and $d\mu_g$ is the volume element on $(M^n,g)$. Stationary points of the bi-energy functional are called biharmonic maps. Jiang (see [J1], [J2]) is the first mathematician who systematically studied the bi-energy functional, and he computed the first and second variations of $E_2$. The stationary points of the functional $E_2$ satisfy the Euler-Lagrange equation
\begin{equation}
-\Delta\tau(\phi)=\sum_{i=1}^n R^{N^m}(d\phi(e_i),\tau(\phi))d\phi(e_i),
\end{equation}
where $\Delta$ is the Laplacian of $(M^n,g)$. Biharmonic submanifolds have attracted a lot of attentions from mathematicians and many important results on biharmonic submanifolds have been obtained since then (see [C1], [C2], [BMO1], [BMO2], [CMO1], [CMO2], [F1], [F2], [F3], [FH]).

The following conjecture was proposed by Balmus-Montaldo-Oniciuc in 2008 [BMO1] (also see Conjecture 7.2 of page 180 in [OC]).

\begin{conjecture}
Any $n$-dimensional biharmonic submanifold in $\mathbb{S}^{n+p}$ has constant mean curvature.
\end{conjecture}

When $n=2, p=1$, the conjecture was proved by Caddeo-Montaldo-Oniciuc in [CMO1]; when  $n=3, p=1$, the conjecture was proved by Balmus-Montaldo-Oniciuc in [BMO2]. In this paper, we prove the conjecture for $n=4, p=1$. In fact, we prove the following theorem:

\begin{theorem}
Four dimensional biharmonic hypersurfaces in nonzero space form $\mathbb{R}^5(c)(c \neq 0)$ have constant mean curvature.
\end{theorem}

In the study of biharmonic submanifolds, there are two other conjectures proposed by Chen in 1991 [C1], and by Caddeo-Montaldo-Oniciuc in 2001 [CMO1], respectively.

\begin{chen conjecture} 
Every $n$-dimensional biharmonic submanifold of Euclidean spaces $\mathbb{R}^{n+p}$ is minimal.
\end{chen conjecture} 

\begin{generalized chen conjecture}
Every $n$-dimensional biharmonic submanifold of a Riemannian manifold $N^{n+p}$ with non-positive sectional curvature is minimal.
\end{generalized chen conjecture}

When $n=2, p=1$, Chen's conjecture was proved by Chen and Jiang around 1987 independently; when $n=3, p=1$, Chen's conjecture was proved by Hasanis and Vlachos in 1995 [HV]. Recently, Fu-Hong-Zhan [FHZ] have made important progress about Chen's conjecture. In fact, they proved Chen's conjecture for $n=4, p=1$.

Ou and Tang [OT] constructed a family of counterexamples, where the generalized Chen's conjecture is false when the ambient space has nonconstant negative sectional curvature. However, the generalized Chen's conjecture remains open when the ambient spaces have constant sectional curvature. In particular, when $p=1, N^{n+1}=\mathbb{H}^{n+1}$, $n=2$ and $n=3$, the generalized Chen's conjecture was proved by Caddeo-Montaldo-Oniciuc [CMO2] and Balmus-Montaldo-Oniciuc [BMO2], respectively. When $c=-1$, our Theorem 1.1 solves the generalized Chen's conjecture for $p=1, N^{n+1}=\mathbb{H}^{n+1}, n=4$.

The paper is organized as follows. In Section 2, we recall some fundamental concepts and formulas for $n$-dimensional biharmonic hypersurfaces in space forms $\mathbb{R}^{n+1}(c)$. In Section 3, for 4-dimensional biharmonic hypersurfaces, we derive some equations and lemmas. In Section 4, we give the proof of Theorem 1.1.

Acknowledgement:  The authors are supported by NSFC-FWO grant No.11961131001. The first two authors are supported by NSFC grant No. 11831005  and No. 11671224.

\section{Preliminaries}

Let $M^n$ be an $n$-dimensional hypersurface in $(n+1)$-dimensional space form $\mathbb{R}^{n+1}(c)$ with constant sectional curvature $c$. When $c=0$, $\mathbb{R}^{n+1}(c)$ is $(n+1)$-dimensional Euclidean space; when $c=1$, $\mathbb{R}^{n+1}(c)$ is $(n+1)$-dimensional unit sphere; when $c=-1$, $\mathbb{R}^{n+1}(c)$ is $(n+1)$-dimensional hyperbolic space. Let $\nabla$ and $\tilde\nabla$ be the Levi-Civita connections of $M^n$ and $\mathbb{R}^{n+1}(c)$. Denote $X$ and $Y$ tangent vector fields of $M^n$ and $\xi$ the unit normal vector field. Then the Gauss formula and Weingarten formula (for example, see [C3]) are

\begin{equation}
\tilde\nabla_X Y=\nabla_X Y+h(X,Y)\xi,
\end{equation}

\begin{equation}
\tilde\nabla_X \xi=-AX,
\end{equation}
where $h$ is the second fundamental form, and $A$ is the Weingarten operator. The mean curvature function $H$ is defined by

\begin{equation}
H=\frac{1}{n} \text { trace } h.
\end{equation}

Moreover, the Gauss and Codazzi equations are given by

\begin{equation}
R(X, Y) Z=c(\langle Y, Z\rangle X-\langle X, Z\rangle Y)+\langle A Y, Z\rangle A X-\langle A X, Z\rangle A Y,
\end{equation}

\begin{equation}
\left(\nabla_{X} A\right) Y=\left(\nabla_{Y} A\right) X,
\end{equation}
where the Riemannian curvature $R(X,Y)Z$ is defined by

\begin{equation}
R(X, Y) Z=\nabla_X \nabla_Y Z-\nabla_Y \nabla_X Z- \nabla_{[X,Y]} Z.
\end{equation}

We have the following characterization result for $M^n$ to be a biharmonic hypersurface in $\mathbb{R}^{n+1}(c)$ (see [J1], [J2], [BMO2], [CMO2], [F3], [FH]).

\begin{proposition}

A hypersurface $M^n$ in a space form $\mathbb{R}^{n+1}(c)$ is biharmonic if and only if $H$ and $A$ satisfy

\begin{equation}
\Delta H+H \text { trace } A^{2}=n c H,
\end{equation}
\begin{equation}
2 A \operatorname{grad} H+n H \operatorname{grad} H=0,
\end{equation}
where the Laplacian operator $\Delta$ acting on a smooth function $f$ on $M^n$ is defined by

\begin{equation}
\Delta f=-\operatorname{div}(\nabla f).
\end{equation}

\end{proposition}

Let $M^n$ be an $n$-dimensional biharmonic hypersurface in $(n+1)$-dimensional space form $\mathbb{R}^{n+1}(c)$. Suppose the mean curvature function $H$ is not constant. From (2.8), we have that $\operatorname{grad} H$ is an eigenvector of the Weingarten operator $A$ with the corresponding principal curvature $-nH/2$. Without loss of generality, we can choose $e_1$ such that $e_1$ is parallel to $\operatorname{grad} H$, and we can choose suitable orthonormal frame $\{e_1,e_2,\cdots,e_n\}$ such that

\begin{equation}
Ae_i=\lambda_i e_i,
\end{equation}
where $\lambda_1=-nH/2$.

Denote the connection coefficients $ \omega_{i j}^{k}$ by

\begin{equation}
\nabla_{e_{i}} e_{j}=\sum_{k=1}^{n} \omega_{i j}^{k} e_{k}, \quad \omega_{i j}^{k}+\omega_{i k}^{j}=0, \quad i, j=1,\cdots,n.
\end{equation}

\begin{lemma} ([FH]) Let $M^n$ be an biharmonic hypersurface in a space form $\mathbb{R}^{n+1}(c)$ and assume the mean curvature $H$ is non-constant. Then the multiplicity of the principal curvature $\lambda_1(=-n H/2)$ is one, that is, $\lambda_j \neq \lambda_1$ for $2 \leq j \leq n$.

\end{lemma}

\begin{lemma} ([FH]) Let $M^n$ be an biharmonic hypersurface in a space form $\mathbb{R}^{n+1}(c)$ and assume the mean curvature $H$ is non-constant. Then the principal curvatures  $\lambda_i$ and the connection coefficients  $\omega_{i i}^{1}$ satisfy
\begin{equation}
e_{1} e_{1}\left(\lambda_{1}\right)=e_{1}\left(\lambda_{1}\right)\left(\sum_{j=2}^{n} \omega_{jj}^{1}\right)+\lambda_{1}\left(n(n-2) c-R+4 \lambda_{1}^{2}\right),
\end{equation}
\begin{equation}
e_{1}\left(\lambda_{i}\right)=(\lambda_{i}-\lambda_{1}) \omega_{i i}^{1}, \quad 2 \leq i \leq n,
\end{equation}
\begin{equation}
e_{1}\left(\omega_{i i}^{1}\right)=\left(\omega_{i i}^{1}\right)^{2}+\lambda_{1} \lambda_{i}+c,  \quad 2 \leq i \leq n,
\end{equation}
where $R$ is the scalar curvature.
\end{lemma}

In 2015, Fu [F2] proved the following theorem.

\begin{theorem}([F2]) 
Let $M^n$ be a biharmonic hypersurface with at most three distinct principal curvatures in $\mathbb{R}^{n+1}(c)$. Then $M^n$ has constant mean curvature.
\end{theorem}

\section{Four dimensional biharmonic hypersurfaces in $\mathbb{R}^{5}(c)$}

From now on, we study the biharmonicity of a hypersurface $M^4$ in a space form $\mathbb{R}^{5}(c)$. By Theorem 2.4, we only need to work on the case that $M^4$ has four distinct principal curvatures, and we assume that the mean curvature $H$ is non-constant. Then there exists a neighborhood of $p$ such that $\operatorname{grad} H \neq 0$. The squared length of the second fundamental form of $M$ is

\begin{equation}
S=\sum_{i=1}^{4} \lambda_{i}^{2}=4 H^{2}+\lambda_{2}^{2}+\lambda_{3}^{2}+\lambda_{4}^{2}.
\end{equation}

By using $\lambda_1=-2H$, Gauss equation is

\begin{equation}
R=12 c+16 H^{2}-S=12 c+12 H^{2}-\lambda_{2}^{2}-\lambda_{3}^{2}-\lambda_{4}^{2}.
\end{equation}

Since $e_1$ is parallel to $\operatorname{grad} H$, it follows that

\begin{equation}
e_{1}(H) \neq 0, \quad e_{2}(H)=e_{3}(H)=e_{4}(H)=0.
\end{equation}

The following result can be found in [F3].

\begin{lemma} ([F3])
Let $M^4$ be a biharmonic hypersurface with four distinct principal curvatures in space forms $\mathbb{R}^{5}(c)$, then we have
\begin{equation}
\nabla_{e_{1}} e_{i} =0, \quad i=1,2,3,4, 
\end{equation}
\begin{equation}
\nabla_{e_{i}} e_{1} =-\omega_{i i}^{1} e_{i}, \quad i=2,3,4, 
\end{equation}
\begin{equation}
\nabla_{e_{i}} e_{i} =\sum_{k=1, k \neq i}^{4} \omega_{i i}^{k} e_{k}, \quad i=2,3,4, 
\end{equation}
\begin{equation}
\nabla_{e_{i}} e_{j} =-\omega_{i i}^{j} e_{i}+\omega_{i j}^{k} e_{k} \quad\text { for distinct } i, j, k=2,3,4,
\end{equation}
where
\begin{equation}
\omega_{i i}^{j}=-\frac{e_{j}\left(\lambda_{i}\right)}{\lambda_{j}-\lambda_{i}}.
\end{equation}
\end{lemma}

Note that we have
\begin{equation}
\lambda_{2}+\lambda_{3}+\lambda_{4}=-3 \lambda_{1},
\end{equation}
\begin{equation}
S=\sum_{i=1}^{4} \lambda_{i}^{2}=\lambda_{1}^{2}+\sum_{i=2}^{4} \lambda_{i}^{2},
\end{equation}
\begin{equation}
e_{1}\left(\lambda_{1}\right) \neq 0, \quad e_{i}\left(\lambda_{1}\right)=0, \quad 2 \leq i \leq 4.
\end{equation}

Following [FH] and [FHZ], the functions $f_k$ are defined by

\begin{equation}
f_k=(\omega_{22}^{1})^k+(\omega_{33}^{1})^k+(\omega_{44}^{1})^k, \quad k=1,\cdots,5.
\end{equation}

The following lemma was proved by Fu-Hong-Zhan in [FHZ]. 

\begin{lemma}([FHZ]) With the notations $f_k$, the following two relations hold

\begin{equation}
f_{1}^{4}-6 f_{1}^{2} f_{2}+3 f_{2}^{2}+8 f_{1} f_{3}-6 f_{4}=0,
\end{equation}

\begin{equation}
f_{1}^{5}-5 f_{1}^{3} f_{2}+5 f_{1}^{2} f_{3}+5 f_{2} f_{3}-6 f_{5}=0.
\end{equation}

\end{lemma}

For simplicity, given a function $g$ on $M^4$, we write $g^\prime = e_1(g), g^{\prime\prime} = e_1e_1(g), g^{\prime\prime\prime} = e_1e_1e_1(g)$ and $g^{\prime\prime\prime\prime} = e_1e_1e_1e_1(g)$. Also, we write $\lambda = \lambda_1$ and $f_1 = T$. Note that for $i=2,3,4$, using Lemma 3.1, we have

\begin{equation}
\begin{aligned}
e_i(\lambda^\prime)= &e_ie_1(\lambda) \\
= & e_1 e_i(\lambda)+[e_i,e_1](\lambda) \\
= & (\nabla_{e_i}e_1-\nabla_{e_1}e_i)(\lambda) =0,
\end{aligned}
\end{equation}
similarly we have 
\begin{equation}
e_i(\lambda^{\prime\prime})=e_i(\lambda^{\prime\prime\prime})=e_i(\lambda^{\prime\prime\prime\prime})=e_i(\lambda^{\prime\prime\prime\prime\prime})=0.
\end{equation}

Following the argument of Fu-Hong-Zhan in [FHZ], we can prove

\begin{lemma} $f_k$ can be written as

\begin{equation}
\left\{\begin{array}{l}
f_{1}=T, \\
 \\
f_{2}=T^{\prime}+3 \lambda^{2}-3 c, \\
 \\
f_{3}=\dfrac{1}{2} T^{\prime \prime}-\left(\lambda^{2}+c\right) T+6 \lambda \lambda^{\prime}, \\
 \\
f_{4}=\dfrac{1}{6} T^{\prime \prime \prime}-\dfrac{4}{3}\left(\lambda^{2}+c\right) T^{\prime}-\dfrac{5}{3} \lambda \lambda^{\prime} T+2 \lambda^{\prime 2}+4 \lambda \lambda^{\prime \prime}-2 \lambda^{4}-c \lambda^{2}+3 c^{2}, \\
 \\
f_{5}=\dfrac{1}{24} T^{\prime \prime \prime \prime}-\dfrac{5}{6}\left(\lambda^{2}+c\right) T^{\prime \prime}+2 \lambda \lambda^{\prime \prime \prime}+\dfrac{5}{3} \lambda^{\prime} \lambda^{\prime \prime}-\dfrac{26}{3} \lambda^{3} \lambda^{\prime}-\dfrac{47}{6} c \lambda \lambda^{\prime} \\
 \\
\qquad -\dfrac{25}{12} \lambda \lambda^{\prime} T^{\prime}-\dfrac{1}{12}\left(13 \lambda \lambda^{\prime \prime}+\lambda^{\prime 2}-12 \lambda^{4}-24 c \lambda^{2}-12 c^{2}\right) T.
\end{array}\right.
\end{equation}

\begin{proof}
Since $e_1(\lambda)\neq 0$, from (2.12), we have
\begin{equation}
f_{1}=\frac{e_{1} e_{1}(\lambda)-\lambda(S-4 c)}{e_{1}(\lambda)}=\frac{\lambda^{\prime \prime}}{\lambda^{\prime}}-\frac{\lambda}{\lambda^{\prime}}(S-4 c)=: T.
\end{equation}
Taking the sum of $i$ from 2 to 4 in (2.14) and (2.13) and using (3.9), we have
\begin{equation}
f_{2}=3 \lambda^{2}+e_{1}\left(f_{1}\right)-3 c=T^{\prime}+3 \lambda^{2}-3 c,
\end{equation}
\begin{equation}
\begin{aligned}
g_{1}: &=\sum_{i=2}^{4} \lambda_{i} \omega_{i i}^{1} \\
&=\lambda T-3 e_{1}(\lambda)=\lambda T-3 \lambda^{\prime}.
\end{aligned}
\end{equation}
Multiplying $\omega_{i i}^{1}$ on both sides of (2.14), we have
\begin{equation*}
\frac{1}{2} e_{1}\left(\left(\omega_{i i}^{1}\right)^{2}\right)=\left(\omega_{i i}^{1}\right)^{3}+\lambda \lambda_{i} \omega_{i i}^{1}+c \omega_{i i}^{1}.
\end{equation*}
Taking the sum of $i$ from 2 to 4, we have
\begin{equation}
\begin{aligned}
f_{3} &=\frac{1}{2} e_{1}\left(f_{2}\right)-\lambda g_{1}-c T \\
&=\frac{1}{2} T^{\prime \prime}-\left(\lambda^{2}+c\right) T+6 \lambda \lambda^{\prime}.
\end{aligned}
\end{equation}
Differentiating (3.20) with respect to $e_1$ and using (2.13) and (2.14), we have
\begin{equation}
e_{1}\left(g_{1}\right)=2 \sum_{i=2}^{4} \lambda_{i}\left(\omega_{i i}^{1}\right)^{2}+\lambda \sum_{i=2}^{4} \lambda_{i}^{2}-\lambda \sum_{i=2}^{4}\left(\omega_{i i}^{1}\right)^{2}-3 c \lambda.
\end{equation}
From (3.9), (3.10), (3.18), we have
\begin{equation*}
\begin{aligned}
g_{2}:=\sum_{i=2}^{4} \lambda_{i}\left(\omega_{i i}^{1}\right)^{2} &=\frac{1}{2}\left\{e_{1}\left(g_{1}\right)-\lambda\left(S-\lambda^{2}\right)+\lambda f_{2}+3 c \lambda\right\} \\
&=\frac{1}{2}\left\{e_{1}\left(g_{1}\right)-\lambda^{\prime \prime}+\lambda^{\prime} T+\lambda^{3}+\lambda f_{2}-c \lambda\right\}.
\end{aligned}
\end{equation*}
Using (3.19) and (3.20), we have
\begin{equation}
g_{2}=\lambda T^{\prime}+\lambda^{\prime} T-2 \lambda^{\prime \prime}+2 \lambda^{3}-2 c \lambda.
\end{equation}
Multiplying $(\omega_{i i}^{1})^2$ on both sides of (2.14), we have
\begin{equation*}
\frac{1}{3} e_{1}\left(\left(\omega_{i i}^{1}\right)^{3}\right)=\left(\omega_{i i}^{1}\right)^{4}+\lambda \lambda_{i}\left(\omega_{i i}^{1}\right)^{2}+c\left(\omega_{i i}^{1}\right)^{2}.
\end{equation*}
Taking the sum of $i$ from 2 to 4, we have
\begin{equation}
\begin{aligned}
f_{4} &=\frac{1}{3} e_{1}\left(f_{3}\right)-\lambda g_{2}-c f_{2} \\
&=\frac{1}{6} T^{\prime \prime \prime}-\frac{4}{3}\left(\lambda^{2}+c\right) T^{\prime}-\frac{5}{3} \lambda \lambda^{\prime} T+2 \lambda^{\prime 2}+4 \lambda \lambda^{\prime \prime}-2 \lambda^{4}-c \lambda^{2}+3 c^{2}.
\end{aligned}
\end{equation}
Multiplying $\lambda_i$ on both sides of (2.13), we have
\begin{equation*}
\lambda_{i}^{2} \omega_{i i}^{1}=\frac{1}{2} e_{1}\left(\lambda_{i}^{2}\right)+\lambda \lambda_{i} \omega_{i i}^{1}.
\end{equation*}
Using (3.10), we have
\begin{equation}
\begin{aligned}
g_{3}:=\sum_{i=2}^{4} \lambda_{i}^{2} \omega_{i i}^{1} &=\frac{1}{2} e_{1}\left(S-\lambda^{2}\right)+\lambda g_{1} \\
&=\frac{1}{2}\left(\frac{\lambda^{\prime \prime}-\lambda^{\prime} T}{\lambda}-\lambda^{2}\right)^{\prime}+\lambda g_{1} \\
&=-\frac{\lambda^{\prime}}{2 \lambda} T^{\prime}+\left(\lambda^{2}-\frac{\lambda^{\prime \prime} \lambda-\lambda^{\prime 2}}{2 \lambda^{2}}\right) T-4 \lambda \lambda^{\prime}+\frac{\lambda^{\prime \prime \prime} \lambda-\lambda^{\prime \prime} \lambda^{\prime}}{2 \lambda^{2}}.
\end{aligned}
\end{equation}
Differentiating (3.23) with respect to $e_1$ and using (2.13) and (2.14), we have
\begin{equation*}
e_{1}\left(g_{2}\right)=3 \sum_{i=2}^{4} \lambda_{i}\left(\omega_{i i}^{1}\right)^{3}-\lambda \sum_{i=2}^{4}\left(\omega_{i i}^{1}\right)^{3}+2 \lambda \sum_{i=2}^{4} \lambda_{i}^{2} \omega_{i i}^{1}+2 c \sum_{i=2}^{4} \lambda_{i} \omega_{i i}^{1},
\end{equation*}
that is
\begin{equation}
\begin{aligned}
g_{4}:=& \sum_{i=2}^{4} \lambda_{i}\left(\omega_{i i}^{1}\right)^{3} \\
=& \frac{1}{3}\left(e_{1}\left(g_{2}\right)+\lambda f_{3}-2 \lambda g_{3}-2 c g_{1}\right) \\
=& \frac{1}{2} \lambda T^{\prime \prime}+\lambda^{\prime} T^{\prime}+\frac{1}{3}\left(2 \lambda^{\prime \prime}-3 \lambda^{3}-\frac{\lambda^{\prime 2}}{\lambda}-3 c \lambda\right) T \\
&-\lambda^{\prime \prime \prime}+\frac{20}{3} \lambda^{2} \lambda^{\prime}+\frac{\lambda^{\prime \prime} \lambda^{\prime}}{3 \lambda}+\frac{4}{3} c \lambda^{\prime}.
\end{aligned}
\end{equation}
Multiplying $(\omega_{i i}^{1})^3$ on both sides of (2.14), we have
\begin{equation*}
\frac{1}{4} e_{1}\left(\left(\omega_{i i}^{1}\right)^{4}\right)=\left(\omega_{i i}^{1}\right)^{5}+\lambda \lambda_{i}\left(\omega_{i i}^{1}\right)^{3}+c\left(\omega_{i i}^{1}\right)^{3}.
\end{equation*}
Taking the sum of $i$ from 2 to 4, we have
\begin{equation}
\begin{aligned}
f_{5}=& \frac{1}{4} e_{1}\left(f_{4}\right)-\lambda g_{4}-c f_{3} \\
=& \frac{1}{24} T^{\prime \prime \prime \prime}-\frac{5}{6}\left(\lambda^{2}+c\right) T^{\prime \prime}+2 \lambda \lambda^{\prime \prime \prime}+\frac{5}{3} \lambda^{\prime} \lambda^{\prime \prime}-\frac{26}{3} \lambda^{3} \lambda^{\prime}-\frac{47}{6} c \lambda \lambda^{\prime} \\
\quad& -\frac{25}{12} \lambda \lambda^{\prime} T^{\prime}-\frac{1}{12}\left(13 \lambda \lambda^{\prime \prime}+\lambda^{\prime 2}-12 \lambda^{4}-24 c \lambda^{2}-12 c^{2}\right) T.
\end{aligned}
\end{equation}
\end{proof}

\end{lemma}

Following the argument of Fu-Hong-Zhan in [FHZ], we can prove
\begin{lemma}
The function $T$ satisfies $e_i(T)=0$ for $i=2,3,4$.
\end{lemma}

\begin{proof}
Assume that $T \neq 0$. Substituting (3.17) into (3.13) and (3.14), we have
\begin{equation}
\begin{aligned}
&9 c^2+10 c T^2+T^4-48 c \lambda^2-26 T^2 \lambda^2+39 \lambda^4\\
-&10 c T^\prime-6 T^2 T^\prime+26 \lambda^2 T^\prime+3 T^{\prime 2}+58 T \lambda \lambda^\prime-12 \lambda^{\prime 2}\\
+&4 T T^{\prime\prime}-24 \lambda \lambda^{\prime\prime}-T^{\prime\prime\prime}=0,
\end{aligned}
\end{equation}
\begin{equation}
\begin{aligned}
&36 c^2 T+40 c T^3+4 T^5-48 c T \lambda^2-80 T^3 \lambda^2-84 T \lambda^4\\
-&20 c T T^\prime-20 T^3 T^\prime-20 T \lambda^2 T^\prime-172 c \lambda \lambda^\prime+120 T^2 \lambda \lambda^\prime+568 \lambda^3 \lambda^\prime\\
+&170 \lambda T^\prime \lambda^\prime+2 T \lambda^{\prime 2}-10 c T^{\prime\prime}+10 T^2 T^{\prime\prime}+50 \lambda^2 T^{\prime\prime}+10 T^\prime T^{\prime\prime}\\
+&26 T \lambda \lambda^{\prime\prime}-40 \lambda^\prime \lambda^{\prime\prime}-48 \lambda \lambda^{\prime\prime\prime}-T^{\prime\prime\prime\prime}=0.
\end{aligned}
\end{equation}
Differentiating (3.28) with respect to $e_1$, we have
\begin{equation}
\begin{aligned}
&20 c T T^\prime+4 T^3 T^\prime-52 T \lambda^2 T^\prime-12 T T^{\prime 2}-96 c \lambda \lambda^\prime-52 T^2 \lambda \lambda^\prime\\
+&156 \lambda^3 \lambda^\prime+110 \lambda T^\prime \lambda^\prime+58 T \lambda^{\prime 2}-10 c T^{\prime\prime}-6 T^2 T^{\prime\prime}+26 \lambda^2 T^{\prime\prime}\\
+&10 T^\prime T^{\prime\prime}+58 T \lambda \lambda^{\prime\prime}-48 \lambda^\prime \lambda^{\prime\prime}+4 T T^{\prime\prime\prime}-24 \lambda \lambda^{\prime\prime\prime}-T^{\prime\prime\prime\prime}=0.
\end{aligned}
\end{equation}
From (3.29) and (3.30), eliminating $T^{\prime\prime\prime\prime}$ we have
\begin{equation}
\begin{aligned}
-&36 c^2 T-40 c T^3-4 T^5+48 c T \lambda^2+80 T^3 \lambda^2+84 T \lambda^4+40 c T T^\prime\\
+&24 T^3 T^\prime-32 T \lambda^2 T^\prime-12 T T^{\prime 2}+76 c \lambda \lambda^\prime-172 T^2 \lambda \lambda^\prime-412 \lambda^3 \lambda^\prime\\
-&60 \lambda T^\prime \lambda^\prime+56 T \lambda^{\prime 2}-16 T^2 T^{\prime\prime}-24 \lambda^2 T^{\prime\prime}+32 T \lambda \lambda^{\prime\prime}-8 \lambda^\prime \lambda^{\prime\prime}\\
+&4 T T^{\prime\prime\prime}+24 \lambda \lambda^{\prime\prime\prime}=0.
\end{aligned}
\end{equation}
From (3.28) and (3.31), eliminating $T^{\prime\prime\prime}$ we have
\begin{equation}
\begin{aligned}
-&36 c T \lambda^2-6 T^3 \lambda^2+60 T \lambda^4+18 T \lambda^2 T^\prime+19 c \lambda \lambda^\prime+15 T^2 \lambda \lambda^\prime\\
-&103 \lambda^3 \lambda^\prime-15 \lambda T^\prime \lambda^\prime+2 T \lambda^{\prime 2}-6 \lambda^2 T^{\prime\prime}-16 T \lambda \lambda^{\prime\prime}-2 \lambda^\prime \lambda^{\prime\prime}+6 \lambda \lambda^{\prime\prime\prime}=0.
\end{aligned}
\end{equation}
Differentiating (3.32) with respect to $e_1$, we have
\begin{equation}
\begin{aligned}
-&36 c \lambda^2 T^\prime-18 T^2 \lambda^2 T^\prime+60 \lambda^4 T^\prime+18 \lambda^2 T^{\prime 2}-72 c T \lambda \lambda^\prime-12 T^3 \lambda \lambda^\prime\\
+&240 T \lambda^3 \lambda^\prime+66 T \lambda T^\prime \lambda^\prime+19 c \lambda^{\prime 2}+15 T^2 \lambda^{\prime 2}-309 \lambda^2 \lambda^{\prime 2}-13 T^\prime \lambda^{\prime 2}\\
+&18 T \lambda^2 T^{\prime\prime}-27 \lambda \lambda^\prime T^{\prime\prime}+19 c \lambda \lambda^{\prime\prime}+15 T^2 \lambda \lambda^{\prime\prime}-103 \lambda^3 \lambda^{\prime\prime}-31 \lambda T^\prime \lambda^{\prime\prime}\\
-&12 T \lambda^\prime \lambda^{\prime\prime}-2 \lambda^{\prime\prime 2}-6 \lambda^2 T^{\prime\prime\prime}-16 T \lambda \lambda^{\prime\prime\prime}+4 \lambda^\prime \lambda^{\prime\prime\prime}+6 \lambda \lambda^{\prime\prime\prime\prime}=0.
\end{aligned}
\end{equation}
From (3.28) and (3.33), eliminating $T^{\prime\prime\prime}$ we have
\begin{equation}
\begin{aligned}
&54 c^2 \lambda^2+60 c T^2 \lambda^2+6 T^4 \lambda^2-288 c \lambda^4-156 T^2 \lambda^4+234 \lambda^6\\
-&24 c \lambda^2 T^\prime-18 T^2 \lambda^2 T^\prime+96 \lambda^4 T^\prime+72 c T \lambda \lambda^\prime+12 T^3 \lambda \lambda^\prime+108 T \lambda^3 \lambda^\prime\\
-&66 T \lambda T^\prime \lambda^\prime-19 c \lambda^{\prime 2}-15 T^2 \lambda^{\prime 2}+237 \lambda^2 \lambda^{\prime 2}+13 T^\prime \lambda^{\prime 2}+6 T \lambda^2 T^{\prime\prime}\\
+&27 \lambda \lambda^\prime T^{\prime\prime}-19 c \lambda \lambda^{\prime\prime}-15 T^2 \lambda \lambda^{\prime\prime}-41 \lambda^3 \lambda^{\prime\prime}+31 \lambda T^\prime \lambda^{\prime\prime}+12 T \lambda^\prime \lambda^{\prime\prime}\\
+&2 \lambda^{\prime\prime 2}+16 T \lambda \lambda^{\prime\prime\prime}-4 \lambda^\prime \lambda^{\prime\prime\prime}-6 \lambda \lambda^{\prime\prime\prime\prime}=0.
\end{aligned}
\end{equation}
From (3.32) and (3.34), eliminating $T^{\prime\prime}$ we have
\begin{equation}
a_{1} T^{\prime}-a_{1} T^{2}+a_{2} T+a_{3}=0,
\end{equation}
where
\begin{equation*}
\begin{aligned}
a_{1}=&62 \lambda^{2} \lambda^{\prime \prime}-109 \lambda \lambda^{\prime 2}+192 \lambda^{5}-48 c \lambda^{3}, \\
a_{2}=&44 \lambda^{2} \lambda^{\prime \prime \prime}-124 \lambda \lambda^{\prime} \lambda^{\prime \prime}+550 \lambda^{4} \lambda^{\prime}+18 \lambda^{\prime 3}-142 c \lambda^{2} \lambda^{\prime}, \\
a_{3}=&-12 \lambda^{2} \lambda^{\prime \prime \prime \prime}+46 \lambda \lambda^{\prime} \lambda^{\prime \prime \prime}-82 \lambda^{4} \lambda^{\prime \prime}+4 \lambda \lambda^{\prime \prime 2} \\
 & -18 \lambda^{\prime 2} \lambda^{\prime \prime}-453 \lambda^{3} \lambda^{\prime 2}+468 \lambda^{7}+108 c^{2} \lambda^{3}-576 c \lambda^{5}+133 c \lambda \lambda^{\prime 2}-38 c \lambda^{2} \lambda^{\prime \prime}.
\end{aligned}
\end{equation*}
\begin{description}
\item[Case (i)] $a_1=0,a_2 \neq 0.$ (3.35) becomes
\end{description}
\begin{equation*}
a_{2} T+a_{3}=0.
\end{equation*}
so we have $T=-\dfrac{a_3}{a_2}$, then $e_i(T)=0$ for $i=2,3,4$. 
\begin{description}
\item[Case (ii)] $a_1=0,a_2=0$. We have
\end{description}
\begin{equation}
62 \lambda \lambda^{\prime\prime} -109  \lambda^{\prime 2} +192 \lambda^4 -48 c \lambda^2=0,
\end{equation}
\begin{equation}
44 \lambda^2 \lambda^{\prime\prime\prime}  -124 \lambda \lambda^\prime \lambda^{\prime\prime}  +550 \lambda^4 \lambda^\prime  +18 \lambda^{\prime 3} -142 c \lambda^2 \lambda^\prime    =0.
\end{equation}
Differentiating (3.36) with respect to $e_1$, from (3.37) we have
\begin{equation}
-1145 c \lambda^2+77 \lambda^4+279 \lambda^{\prime 2}-206 \lambda \lambda^{\prime\prime}=0.
\end{equation}
From (3.36) and (3.38), eliminating $\lambda^{\prime\prime}$ we have
\begin{equation}
-40439 c \lambda^2+22163 \lambda^4-2578 \lambda^{\prime 2}=0.
\end{equation}
Differentiating (3.39) with respect to $e_1$, we have
\begin{equation}
-40439 c \lambda+44326 \lambda^3-2578 \lambda^{\prime\prime}=0.
\end{equation}
Substituting (3.39) and (3.40) into (3.36), we have
\begin{equation*}
1776889 c + 827421 \lambda^2 = 0,
\end{equation*}
which is a contradiction. 
\begin{description}
\item[Case (iii)] $a_1 \neq 0$. Differentiating (3.35) with respect to $e_1$, we have
\end{description}
\begin{equation}
\begin{aligned}
&96 c T \lambda^3 T^\prime-384 T \lambda^5 T^\prime+324 c^2 \lambda^2 \lambda^\prime+144 c T^2 \lambda^2 \lambda^\prime-2880 c \lambda^4 \lambda^\prime \\
-&960 T^2 \lambda^4 \lambda^\prime+3276 \lambda^6 \lambda^\prime-286 c \lambda^2 T^\prime \lambda^\prime+1510 \lambda^4 T^\prime \lambda^\prime-284 c T \lambda \lambda^{\prime 2} \\
+&2200 T \lambda^3 \lambda^{\prime 2}+218 T \lambda T^\prime \lambda^{\prime 2}+133 c \lambda^{\prime 3}+109 T^2 \lambda^{\prime 3}-1359 \lambda^2 \lambda^{\prime 3}\\
-&91 T^\prime \lambda^{\prime 3} -48 c \lambda^3 T^{\prime\prime}+192 \lambda^5 T^{\prime\prime}-109 \lambda \lambda^{\prime 2} T^{\prime\prime}-142 c T \lambda^2 \lambda^{\prime\prime}\\
+&550 T \lambda^4 \lambda^{\prime\prime}-124 T \lambda^2 T^\prime \lambda^{\prime\prime} +190 c \lambda \lambda^\prime \lambda^{\prime\prime}+94 T^2 \lambda \lambda^\prime \lambda^{\prime\prime}-1234 \lambda^3 \lambda^\prime \lambda^{\prime\prime}\\
-&218 \lambda T^\prime \lambda^\prime \lambda^{\prime\prime}-70 T \lambda^{\prime 2} \lambda^{\prime\prime} +62 \lambda^2 T^{\prime\prime} \lambda^{\prime\prime}-124 T \lambda \lambda^{\prime\prime 2}-32 \lambda^\prime \lambda^{\prime\prime 2} \\
-&38 c \lambda^2 \lambda^{\prime\prime\prime}-62 T^2 \lambda^2 \lambda^{\prime\prime\prime} -82 \lambda^4 \lambda^{\prime\prime\prime}+106 \lambda^2 T^\prime \lambda^{\prime\prime\prime}-36 T \lambda \lambda^\prime \lambda^{\prime\prime\prime}\\
+&28 \lambda^{\prime 2} \lambda^{\prime\prime\prime}+54 \lambda \lambda^{\prime\prime} \lambda^{\prime\prime\prime} +44 T \lambda^2 \lambda^{\prime\prime\prime\prime}+22 \lambda \lambda^\prime \lambda^{\prime\prime\prime\prime}-12 \lambda^2 \lambda^{\prime\prime\prime\prime\prime}=0.
\end{aligned}
\end{equation}
From (3.32) and (3.41), eliminating $T^{\prime\prime}$ we have
\begin{equation}
\begin{aligned}
&1728 c^2 T \lambda^4+288 c T^3 \lambda^4-9792 c T \lambda^6-1152 T^3 \lambda^6\\
+&11520 T \lambda^8-288 c T \lambda^4 T^\prime+1152 T \lambda^6 T^\prime+1032 c^2 \lambda^3 \lambda^\prime\\
+&144 c T^2 \lambda^3 \lambda^\prime-8688 c \lambda^5 \lambda^\prime-2880 T^2 \lambda^5 \lambda^\prime-120 \lambda^7 \lambda^\prime\\
-&996 c \lambda^3 T^\prime \lambda^\prime+6180 \lambda^5 T^\prime \lambda^\prime+2124 c T \lambda^2 \lambda^{\prime 2}+654 T^3 \lambda^2 \lambda^{\prime 2}\\
+&7044 T \lambda^4 \lambda^{\prime 2}-654 T \lambda^2 T^\prime \lambda^{\prime 2}-1273 c \lambda \lambda^{\prime 3}-981 T^2 \lambda \lambda^{\prime 3}\\
+&3073 \lambda^3 \lambda^{\prime 3}+1089 \lambda T^\prime \lambda^{\prime 3}-218 T \lambda^{\prime 4}-2316 c T \lambda^3 \lambda^{\prime\prime}\\
-&372 T^3 \lambda^3 \lambda^{\prime\prime}+3948 T \lambda^5 \lambda^{\prime\prime}+372 T \lambda^3 T^\prime \lambda^{\prime\prime}+2414 c \lambda^2 \lambda^\prime \lambda^{\prime\prime}\\
+&1494 T^2 \lambda^2 \lambda^\prime \lambda^{\prime\prime}-14174 \lambda^4 \lambda^\prime \lambda^{\prime\prime}-2238 \lambda^2 T^\prime \lambda^\prime \lambda^{\prime\prime}+1448 T \lambda \lambda^{\prime 2} \lambda^{\prime\prime}\\
+&218 \lambda^{\prime 3} \lambda^{\prime\prime}-1736 T \lambda^2 \lambda^{\prime\prime 2}-316 \lambda \lambda^\prime \lambda^{\prime\prime 2}-516 c \lambda^3 \lambda^{\prime\prime\prime}\\
-&372 T^2 \lambda^3 \lambda^{\prime\prime\prime}+660 \lambda^5 \lambda^{\prime\prime\prime}+636 \lambda^3 T^\prime \lambda^{\prime\prime\prime}-216 T \lambda^2 \lambda^\prime \lambda^{\prime\prime\prime}\\
-&486 \lambda \lambda^{\prime 2} \lambda^{\prime\prime\prime}+696 \lambda^2 \lambda^{\prime\prime} \lambda^{\prime\prime\prime}+264 T \lambda^3 \lambda^{\prime\prime\prime\prime}+132 \lambda^2 \lambda^\prime \lambda^{\prime\prime\prime\prime}-72 \lambda^3 \lambda^{\prime\prime\prime\prime\prime}=0.
\end{aligned}
\end{equation}
From (3.35) and (3.42), eliminating $T^{\prime}$ we have
\begin{equation}
b_1 T+b_2=0,
\end{equation}
where
\begin{equation*}
\begin{aligned}
b_1=-&6480 c^3 \lambda^6+63936 c^2 \lambda^8-204336 c \lambda^{10}+209088 \lambda^{12}\\
-&40350 c^2 \lambda^4 \lambda^{\prime 2}+237750 c \lambda^6 \lambda^{\prime 2}-309288 \lambda^8 \lambda^{\prime 2}+4812 c \lambda^2 \lambda^{\prime 4}\\
-&227013 \lambda^4 \lambda^{\prime 4}+520 \lambda^{\prime 6}+20898 c^2 \lambda^5 \lambda^{\prime\prime}-125856 c \lambda^7 \lambda^{\prime\prime}\\
+&174078 \lambda^9 \lambda^{\prime\prime}-25773 c \lambda^3 \lambda^{\prime 2} \lambda^{\prime\prime}+302157 \lambda^5 \lambda^{\prime 2} \lambda^{\prime\prime}-975 \lambda \lambda^{\prime 4} \lambda^{\prime\prime}\\
-&5622 c \lambda^4 \lambda^{\prime\prime 2}-7830 \lambda^6 \lambda^{\prime\prime 2}+1350 \lambda^2 \lambda^{\prime 2} \lambda^{\prime\prime 2}-13640 \lambda^3 \lambda^{\prime\prime 3}\\
+&19719 c \lambda^4 \lambda^\prime \lambda^{\prime\prime\prime}-89523 \lambda^6 \lambda^\prime \lambda^{\prime\prime\prime}-717 \lambda^2 \lambda^{\prime 3} \lambda^{\prime\prime\prime}+18354 \lambda^3 \lambda^\prime \lambda^{\prime\prime} \lambda^{\prime\prime\prime}\\
-&3498 \lambda^4 \lambda^{\prime\prime\prime 2}-2016 c \lambda^5 \lambda^{\prime\prime\prime\prime}+8064 \lambda^7 \lambda^{\prime\prime\prime\prime}-4578 \lambda^3 \lambda^{\prime 2} \lambda^{\prime\prime\prime\prime}\\
+&2604 \lambda^4 \lambda^{\prime\prime} \lambda^{\prime\prime\prime\prime},
\end{aligned}
\end{equation*}
\begin{equation*}
\begin{aligned}
b_2=\quad&7254 c^3 \lambda^5 \lambda^\prime-78246 c^2 \lambda^7 \lambda^\prime+295434 c \lambda^9 \lambda^\prime-364410 \lambda^{11} \lambda^\prime\\
-&4566 c^2 \lambda^3 \lambda^{\prime 3}-11349 c \lambda^5 \lambda^{\prime 3}+361623 \lambda^7 \lambda^{\prime 3}-760 c \lambda \lambda^{\prime 5}\\
+&19795 \lambda^3 \lambda^{\prime 5}+18996 c^2 \lambda^4 \lambda^\prime \lambda^{\prime\prime}-66342 c \lambda^6 \lambda^\prime \lambda^{\prime\prime}-146838 \lambda^8 \lambda^\prime \lambda^{\prime\prime}\\
-&3926 c \lambda^2 \lambda^{\prime 3} \lambda^{\prime\prime}+120509 \lambda^4 \lambda^{\prime 3} \lambda^{\prime\prime}-520 \lambda^{\prime 5} \lambda^{\prime\prime}+10472 c \lambda^3 \lambda^\prime \lambda^{\prime\prime 2}\\
-&143462 \lambda^5 \lambda^\prime \lambda^{\prime\prime 2}+415 \lambda \lambda^{\prime 3} \lambda^{\prime\prime 2}-1330 \lambda^2 \lambda^\prime \lambda^{\prime\prime 3}-5490 c^2 \lambda^5 \lambda^{\prime\prime\prime}\\
+&29448 c \lambda^7 \lambda^{\prime\prime\prime}-21366 \lambda^9 \lambda^{\prime\prime\prime}+5100 c \lambda^3 \lambda^{\prime 2} \lambda^{\prime\prime\prime}-20178 \lambda^5 \lambda^{\prime 2} \lambda^{\prime\prime\prime}\\
+&360 \lambda \lambda^{\prime 4} \lambda^{\prime\prime\prime}-5154 c \lambda^4 \lambda^{\prime\prime} \lambda^{\prime\prime\prime}+28338 \lambda^6 \lambda^{\prime\prime} \lambda^{\prime\prime\prime}+1050 \lambda^2 \lambda^{\prime 2} \lambda^{\prime\prime} \lambda^{\prime\prime\prime}\\
+&5076 \lambda^3 \lambda^{\prime\prime 2} \lambda^{\prime\prime\prime}-3657 \lambda^3 \lambda^\prime \lambda^{\prime\prime\prime 2}-2286 c \lambda^4 \lambda^\prime \lambda^{\prime\prime\prime\prime}+12438 \lambda^6 \lambda^\prime \lambda^{\prime\prime\prime\prime}\\
-&165 \lambda^2 \lambda^{\prime 3} \lambda^{\prime\prime\prime\prime}-2334 \lambda^3 \lambda^\prime \lambda^{\prime\prime} \lambda^{\prime\prime\prime\prime}+954 \lambda^4 \lambda^{\prime\prime\prime} \lambda^{\prime\prime\prime\prime}+432 c \lambda^5 \lambda^{\prime\prime\prime\prime\prime}\\
-&1728 \lambda^7 \lambda^{\prime\prime\prime\prime\prime}+981 \lambda^3 \lambda^{\prime 2} \lambda^{\prime\prime\prime\prime\prime}-558 \lambda^4 \lambda^{\prime\prime} \lambda^{\prime\prime\prime\prime\prime}.
\end{aligned}
\end{equation*}
If $b_1 \neq 0$, then $T=-\dfrac{b_2}{b_1}$, and $e_i(T)=0$ for $i=2,3,4$. If $b_1=0$, then $b_2=0$. Using the similar technique in [FHZ] and (3.15), (3.16), we can eliminate $\lambda^{\prime\prime\prime\prime\prime},\lambda^{\prime\prime\prime\prime},\lambda^{\prime\prime\prime},\lambda^{\prime\prime},\lambda^{\prime}$ and get a nontrivial polynomial of $\lambda$ with constant coefficients. Thus $\lambda$ is a constant, which is a contradiction.
\end{proof}

Following the argument of Fu-Hong-Zhan in [FHZ], we can get
\begin{lemma}

Suppose $M^4$ has four distinct principal curvatures, then $e_j(\lambda_i)=0$ for $2 \leq i, j \leq 4$. 

\end{lemma}

Moreover, from (3.8) we have $\omega_{i i}^j = 0$ for  $2 \leq i, j \leq 4$.

\section{Proof of Theorem 1.1}

We need the following two lemmas proved by Fu and Hong in [FH].

\begin{lemma}([FH])

\begin{equation}
\omega_{23}^{4}\left(\lambda_{3}-\lambda_{4}\right)=\omega_{32}^{4}\left(\lambda_{2}-\lambda_{4}\right)=\omega_{43}^{2}\left(\lambda_{3}-\lambda_{2}\right),
\end{equation}

\begin{equation}
\omega_{23}^{4} \omega_{32}^{4}+\omega_{34}^{2} \omega_{43}^{2}+\omega_{24}^{3} \omega_{42}^{3}=0,
\end{equation}

\begin{equation}
\omega_{23}^{4}\left(\omega_{33}^{1}-\omega_{44}^{1}\right)=\omega_{32}^{4}\left(\omega_{22}^{1}-\omega_{44}^{1}\right)=\omega_{43}^{2}\left(\omega_{33}^{1}-\omega_{22}^{1}\right).
\end{equation}

\end{lemma}

\begin{lemma}([FH])

\begin{equation}
\omega_{22}^{1} \omega_{33}^{1}-2 \omega_{23}^{4} \omega_{32}^{4}=-\lambda_{2} \lambda_{3}-c,
\end{equation}

\begin{equation}
\omega_{22}^{1} \omega_{44}^{1}-2 \omega_{24}^{3} \omega_{42}^{3}=-\lambda_{2} \lambda_{4}-c,
\end{equation}

\begin{equation}
\omega_{33}^{1} \omega_{44}^{1}-2 \omega_{34}^{2} \omega_{43}^{2}=-\lambda_{3} \lambda_{4}-c.
\end{equation}

\end{lemma}

\begin{description}
\item[The proof of Theorem 1.1]  
\end{description}
Let $M^4$ be a biharmonic hypersurface in $\mathbb{R}^{5}(c)(c \neq 0)$. There exists a smooth function $a$ such that
\begin{equation}
\omega_{23}^{4}=a(\lambda_2-\lambda_3)(\lambda_2-\lambda_4).
\end{equation}
From (4.1) and (4.7) we have
\begin{equation}
\omega_{34}^{2}=a(\lambda_3-\lambda_4)(\lambda_3-\lambda_2),
\end{equation}
\begin{equation}
\omega_{42}^{3}=a(\lambda_4-\lambda_2)(\lambda_4-\lambda_3).
\end{equation}
Taking $X=e_1,Y=e_2,Z=e_3$ in Gauss equation (2.4), we have
\begin{equation}
\begin{aligned}
e_1(a) = -&\frac{a}{3(\lambda_2-\lambda_3)(\lambda_2-\lambda_4)}\{(5 \lambda_2^2+\lambda_2 \lambda_3-\lambda_3^2+\lambda_2 \lambda_4-5 \lambda_3 \lambda_4-\lambda_4^2)\omega_{22}^{1} \\
 +& (-\lambda_2^2 - 4 \lambda_2 \lambda_3 + 4 \lambda_3 \lambda_4 + \lambda_4^2)\omega_{33}^{1}  \\
 +& (-\lambda_2^2 + \lambda_3^2 - 4 \lambda_2 \lambda_4 + 4 \lambda_3 \lambda_4)\omega_{44}^{1}\}.
\end{aligned}
\end{equation}
Taking $X=e_1,Y=e_3,Z=e_2$  in Gauss equation (2.4), we have
\begin{equation}
\begin{aligned}
e_1(a) = -&\frac{a}{3(\lambda_2-\lambda_3)(\lambda_3-\lambda_4)}\{(4 \lambda_2 \lambda_3 + \lambda_3^2 - 4 \lambda_2 \lambda_4 - \lambda_4^2)\omega_{22}^{1} \\
 +& (\lambda_2^2 - \lambda_2 \lambda_3 - 5 \lambda_3^2 + 5 \lambda_2 \lambda_4 - \lambda_3 \lambda_4 + \lambda_4^2)\omega_{33}^{1}  \\
 +& (-\lambda_2^2 + \lambda_3^2 - 4 \lambda_2 \lambda_4 + 4 \lambda_3 \lambda_4)\omega_{44}^{1}\}.
\end{aligned}
\end{equation}
Taking  $X=e_1,Y=e_4,Z=e_2$ in Gauss equation (2.4), we have
\begin{equation}
\begin{aligned}
e_1(a) = -&\frac{a}{3(\lambda_2-\lambda_4)(\lambda_3-\lambda_4)}\{(4 \lambda_2 \lambda_3 + \lambda_3^2 - 4 \lambda_2 \lambda_4 - \lambda_4^2)\omega_{22}^{1} \\
 +& (\lambda_2^2 + 4 \lambda_2 \lambda_3 - 4 \lambda_3 \lambda_4 - \lambda_4^2)\omega_{33}^{1}  \\
 +& (-\lambda_2^2 - 5 \lambda_2 \lambda_3 - \lambda_3^2 + \lambda_2 \lambda_4 + \lambda_3 \lambda_4 + 5 \lambda_4^2)\omega_{44}^{1}\}.
\end{aligned}
\end{equation}
(4.10)+(4.11)+(4.12) implies
\begin{equation}
e_1(a)=\frac{a}{9(\lambda_2-\lambda_3)(\lambda_2-\lambda_4)(\lambda_3-\lambda_4)}(k_2\omega_{22}^{1} -k_3\omega_{33}^{1} +k_4\omega_{44}^{1}),
\end{equation}
where
\begin{equation*}
\begin{aligned}
k_2=&(\lambda_3 - \lambda_4) (-13 \lambda_2^2 + 2 \lambda_3^2 + 7 \lambda_3 \lambda_4 + 2 \lambda_4^2 + \lambda_2 \lambda_3 + \lambda_2 \lambda_4),\\
k_3=&(\lambda_2 - \lambda_4) (2 \lambda_2^2 - 13 \lambda_3^2 + \lambda_3 \lambda_4 + 2 \lambda_4^2 + \lambda_2 \lambda_3 + 7 \lambda_2 \lambda_4),\\
k_4=&(\lambda_2 - \lambda_3) (2 \lambda_2^2 + 7 \lambda_2 \lambda_3 + 2 \lambda_3^2 + \lambda_2 \lambda_4 + \lambda_3 \lambda_4 - 13 \lambda_4^2).
\end{aligned}
\end{equation*}
Taking $X=e_2,Y=e_4,Z=e_1$ in Gauss equation (2.4), we have
\begin{equation}
a\{(\lambda_3-\lambda_4)\omega_{22}^{1}-(\lambda_2-\lambda_4)\omega_{33}^{1}+(\lambda_2-\lambda_3)\omega_{44}^{1}\}=0.
\end{equation}
We can rewrite the biharmonic equation (2.7) as
\begin{equation}
-e_1e_1(\lambda_1)+e_1(\lambda_1)(\omega_{22}^{1}+\omega_{33}^{1}+\omega_{44}^{1})+\lambda_1(8c+4\lambda_1^2-R)=0.
\end{equation}
Taking $X=e_2,Y=e_4,Z=e_2$ in Gauss equation (2.4), we have $e_2(a)=0$ and
\begin{equation}
\omega_{22}^{1}\omega_{44}^{1}+2 a^2 (\lambda_2 - \lambda_3) (\lambda_2 - \lambda_4)^2 (\lambda_3 - \lambda_4) + \lambda_2 \lambda_4 + c=0.
\end{equation}
By symmetry, we have $e_3(a)=e_4(a)=0$ and
\begin{equation}
\omega_{33}^{1}\omega_{44}^{1}-2 a^2 (\lambda_2 - \lambda_3) (\lambda_2 - \lambda_4) (\lambda_3 - \lambda_4)^2 + \lambda_3 \lambda_4 + c=0,
\end{equation}
\begin{equation}
\omega_{22}^{1}\omega_{33}^{1}+\lambda_2 \lambda_3 - 2 a^2 (\lambda_2 - \lambda_3)^2 (\lambda_2 - \lambda_4) (\lambda_3 - \lambda_4) + c=0.
\end{equation}
Here we introduce the new variables $y_1,y_2,y_3$ by
\begin{equation}
y_1=\lambda_2+\lambda_3+\lambda_4=-3\lambda_1,
\end{equation}
\begin{equation}
y_2=\lambda_2 \lambda_3+\lambda_2 \lambda_4+\lambda_3 \lambda_4,
\end{equation}
\begin{equation}
y_3=\lambda_2 \lambda_3 \lambda_4.
\end{equation}
Next we consider two cases.
\begin{description}
\item[Case A] $a \neq 0$. From (4.14) we have
\begin{equation}
(\lambda_3-\lambda_4)\omega_{22}^{1}-(\lambda_2-\lambda_4)\omega_{33}^{1}+(\lambda_2-\lambda_3)\omega_{44}^{1}=0.
\end{equation}
\end{description}
From (4.1) and (4.3), there exists smooth functions $\kappa$ and $\tau$ such that
\begin{equation}
\omega_{ii}^{1}=\kappa \lambda_i + \tau, \quad\quad i=2,3,4.
\end{equation}
From (2.13) and (2.14), we have
\begin{equation}
e_1(\kappa)=-\frac{1}{3} (1 + \kappa^2) (\lambda_2 + \lambda_3 + \lambda_4) + \kappa \tau,
\end{equation}
\begin{equation}
e_1(\tau)=c-\frac{1}{3}\kappa \tau (\lambda_2 + \lambda_3 + \lambda_4)+\tau^2.
\end{equation}
(4.16)+(4.17)+(4.18) implies
\begin{equation}
3c+(1+\kappa^2)y_2+2\kappa\tau y_1+3\tau^2=0.
\end{equation}
From (4.26) we can solve for $y_2$ 
\begin{equation}
y_2=-\frac{1}{1+\kappa^2}(3c+2\kappa\tau y_1+3\tau^2).
\end{equation}
From (2.13), (2.14), (4.23), we have
\begin{equation}
e_1(y_1)=\frac{4}{3}\kappa y_1^2-2\kappa y_2+2\tau y_1.
\end{equation}
From (4.24), (4.25), (4.27), (4.28), we have
\begin{equation}
\begin{aligned}
e_1(y_1^2-2y_2)=&\frac{4}{3(1+\kappa^2)^2}(2 y_1^3 \kappa (1 + \kappa^2)^2 + 15 c y_1 (\kappa + \kappa^3) + 9 c (1 + 2 \kappa^2) \tau \\
+& y_1^2 (1 + \kappa^2) (2 + 13 \kappa^2) \tau + 3 y_1 \kappa (7 + 9 \kappa^2) \tau^2 + 9 (1 + 2 \kappa^2) \tau^3).
\end{aligned}
\end{equation}
From (2.13), (2.14), (4.23), we have
\begin{equation}
e_1(\lambda_2^2+\lambda_3^2+\lambda_4^2)=\frac{2}{3}(4 y_1^3 \kappa - 11 y_1 y_2 \kappa + 9 y_3 \kappa + 4 y_1^2 \tau - 6 y_2 \tau).
\end{equation}
From (4.29) and (4.30), we have
\begin{equation}
-c y_1 - 3 y_3 - c y_1 \kappa^2 - 6 y_3 \kappa^2 - 3 y_3 \kappa^4 +  6 c \kappa \tau - y_1 \tau^2 + 3 y_1 \kappa^2 \tau^2 +  6 \kappa \tau^3=0.
\end{equation}
From (4.31), we can solve for $y_3$ 
\begin{equation}
y_3=\frac{1}{3(1+\kappa^2)^2}(-c(1+\kappa^2)y_1+6c\kappa\tau+(3\kappa^2-1)\tau^2 y_1+6\kappa \tau^3).
\end{equation}
Using (4.23), (4.15) can be rewritten as
\begin{equation}
27 c y_1 - 7 y_1^3 + 12 y_1 y_2 + 8 y_1^3 \kappa^2 - 27 y_1 y_2 \kappa^2 + 27 y_3 \kappa^2 + 6 y_1^2 \kappa \tau - 18 y_2 \kappa \tau = 0.
\end{equation}
Differentiating (4.33) with respect to $e_1$, we have
\begin{equation}
\begin{aligned}
-&684 c y_1^2 \kappa - 100 y_1^4 \kappa + 2106 c^2 \kappa^3 + 324 c y_1^2 \kappa^3 \\
-& 120 y_1^4 \kappa^3 + 1008 c y_1^2 \kappa^5 + 60 y_1^4 \kappa^5 + 80 y_1^4 \kappa^7 \\
-& 324 c y_1 \tau - 108 y_1^3 \tau + 648 c y_1 \kappa^2 \tau - 618 y_1^3 \kappa^2 \tau \\
+& 4806 c y_1 \kappa^4 \tau + 330 y_1^3 \kappa^4 \tau + 840 y_1^3 \kappa^6 \tau - 1062 y_1^2 \kappa \tau^2 \\
+& 5346 c \kappa^3 \tau^2 + 558 y_1^2 \kappa^3 \tau^2 + 3240 y_1^2 \kappa^5 \tau^2 - 486 y_1 \tau^3 \\
+& 324 y_1 \kappa^2 \tau^3 + 5400 y_1 \kappa^4 \tau^3 + 3240 \kappa^3 \tau^4 = 0.
\end{aligned}
\end{equation}
Differentiating (4.34) with respect to $e_1$, we have
\begin{equation}
\begin{aligned}
-&972 c^2 y_1 + 360 c y_1^3 + 100 y_1^5 - 29970 c^2 y_1 \kappa^2 \\
-& 14454 c y_1^3 \kappa^2 - 1040 y_1^5 \kappa^2 + 15390 c^2 y_1 \kappa^4 - 18684 c y_1^3 \kappa^4 \\
-& 3000 y_1^5 \kappa^4 + 44388 c^2 y_1 \kappa^6 + 7434 c y_1^3 \kappa^6 - 1760 y_1^5 \kappa^6 \\
+& 11304 c y_1^3 \kappa^8 + 820 y_1^5 \kappa^8 + 720 y_1^5 \kappa^{10} - 5832 c^2 \kappa \tau \\
-& 24732 c y_1^2 \kappa \tau - 2652 y_1^4 \kappa \tau + 62694 c^2 \kappa^3 \tau - 77112 c y_1^2 \kappa^3 \tau \\
-& 18294 y_1^4 \kappa^3 \tau + 137538 c^2 \kappa^5 \tau + 59022 c y_1^2 \kappa^5 \tau - 16992 y_1^4 \kappa^5 \tau \\
+& 111402 c y_1^2 \kappa^7 \tau + 10290 y_1^4 \kappa^7 \tau + 11640 y_1^4 \kappa^9 \tau - 7290 c y_1 \tau^2 \\
-& 1206 y_1^3 \tau^2 - 77436 c y_1 \kappa^2 \tau^2 - 35964 y_1^3 \kappa^2 \tau^2 + 158922 c y_1 \kappa^4 \tau^2 \\
-& 58194 y_1^3 \kappa^4 \tau^2 + 333396 c y_1 \kappa^6 \tau^2 + 48564 y_1^3 \kappa^6 \tau^2 + 72000 y_1^3 \kappa^8 \tau^2 \\
-& 14580 c \kappa \tau^3 - 29268 y_1^2 \kappa \tau^3 + 136566 c \kappa^3 \tau^3 - 89154 y_1^2 \kappa^3 \tau^3 \\
+& 302778 c \kappa^5 \tau^3 + 115074 y_1^2 \kappa^5 \tau^3 + 213840 y_1^2 \kappa^7 \tau^3 - 7290 y_1 \tau^4 \\
-& 54270 y_1 \kappa^2 \tau^4 + 144180 y_1 \kappa^4 \tau^4 + 304560 y_1 \kappa^6 \tau^4 - 8748 \kappa \tau^5 \\
+& 73872 \kappa^3 \tau^5 + 165240 \kappa^5 \tau^5 = 0.
\end{aligned}
\end{equation}
From (4.33) and (4.34), eliminating $\tau$ we have
\begin{equation}
\sum_{m=0}^{8}P_{2m} \kappa^{2m} = 0,
\end{equation}
where
\newpage
\begin{equation*}
\begin{aligned}
P_0 = -&164025 c^3 y_1^4 - 149445 c^2 y_1^6 - 17739 c y_1^8 - 567 y_1^{10},\\
P_2 = -& 157464 c^4 y_1^2  + 988524 c^3 y_1^4  - 1879848 c^2 y_1^6  + 449388 c y_1^8  - 18072 y_1^{10},\\
P_4 = \quad& 6114852 c^4 y_1^2  + 8627715 c^3 y_1^4  - 4223745 c^2 y_1^6 + 79353 c y_1^8  - 8223 y_1^{10},\\
P_6 = \quad& 6141096 c^5  - 46189440 c^4 y_1^2 + 84187836 c^3 y_1^4  - 36738684 c^2 y_1^6  \\
+& 4146300 c y_1^8  - 172964 y_1^{10},\\
P_8 = \quad& 6141096 c^5  - 223047756 c^4 y_1^2  + 137387340 c^3 y_1^4  - 83992140 c^2 y_1^6 \\
+& 12535884 c y_1^8  - 488104 y_1^{10},  \\
P_{10} = -& 35311302 c^5  + 54154494 c^4 y_1^2 - 49487436 c^3 y_1^4  - 69538500 c^2 y_1^6  \\
+& 11334546 c y_1^8  - 413882 y_1^{10},  \\
P_{12} = -& 18423288 c^5  + 1192107456 c^4 y_1^2  - 355203792 c^3 y_1^4  - 6360768 c^2 y_1^6  \\
+& 3415752 c y_1^8  - 121088 y_1^{10},  \\
P_{14} = \quad&  55269864 c^5  + 1533226968 c^4 y_1^2 - 345382704 c^3 y_1^4  + 18635184 c^2 y_1^6  \\
+& 46152 c y_1^8  - 15496 y_1^{10}, \\
P_{16} = \quad&  587865600 c^4 y_1^2  - 105629184 c^3 y_1^4  + 5505408 c^2 y_1^6  - 36864 c y_1^8  - 2432 y_1^{10}.  
\end{aligned}
\end{equation*}
From (4.33) and (4.35), eliminating $\tau$ we have
\begin{equation}
\sum_{m=0}^{13}Q_{2m} \kappa^{2m} = 0,
\end{equation}
where
\begin{equation*}
\begin{aligned}
Q_0= \quad& 89813529 c^4 y_1^4 + 623321244 c^3 y_1^6 + 1108270998 c^2 y_1^8 \\
+& 92935836 c y_1^{10} + 1996569 y_1^{12},\\
Q_2= \quad& 7620155352 c^5 y_1^2  + 18958390038 c^4 y_1^4  - 2051878392 c^3 y_1^6  \\
+& 18386833140 c^2 y_1^8  - 3132184464 c y_1^{10}  + 199117734 y_1^{12},\\
Q_4=  \quad& 2550916800 c^6  - 3713237316 c^5 y_1^2  + 396683144775 c^4 y_1^4  \\
-& 281478327804 c^3 y_1^6  + 162873940914 c^2 y_1^8  - 23218046136 c y_1^{10}  \\
+& 1170436927 y_1^{12},\\
Q_6= \quad& 21533989320 c^6  - 253030869900 c^5 y_1^2  + 2983350761604 c^4 y_1^4  \\
-& 2233514241576 c^3 y_1^6  + 897578511552 c^2 y_1^8  - 110846416140 c y_1^{10}  \\
+& 5032123316 y_1^{12},\\
Q_8= \quad& 169801776792 c^6  - 4722853102668 c^5 y_1^2  + 19054644126723 c^4 y_1^4  \\
-& 11292218972532 c^3 y_1^6  + 3695744541258 c^2 y_1^8  - 455003760600 c y_1^{10}  \\
+& 19351326643 y_1^{12},\\
Q_{10}= \quad& 554978521890 c^6  - 30549126042468 c^5 y_1^2  + 83834398712052 c^4 y_1^4 \\ 
-& 42196006098576 c^3 y_1^6   + 11288209375170 c^2 y_1^8  - 1281849786108 c y_1^{10}  \\
+& 48941467416 y_1^{12},\\
\end{aligned}
\end{equation*}
\newpage
\begin{equation*}
\begin{aligned}
Q_{12}=-& 310363669764 c^6  - 57933464951484 c^5 y_1^2  + 211026987943857 c^4 y_1^4  \\
-& 105270974689716 c^3 y_1^6  + 23637191018346 c^2 y_1^8  - 2258485653816 c y_1^{10}  \\
+& 72974080657 y_1^{12},\\
Q_{14}=-& 4681635955884 c^6  + 135848247734196 c^5 y_1^2  + 275115121235820 c^4 y_1^4  \\
-& 165192051148632 c^3 y_1^6  + 32239821679308 c^2 y_1^8  - 2470452821820 c y_1^{10}  \\
+& 62750560596 y_1^{12},\\
Q_{16}=-& 5901463112004 c^6  + 844079357467764 c^5 y_1^2  + 106282529065260 c^4 y_1^4  \\
-& 157860516846024 c^3 y_1^6  + 27762314926644 c^2 y_1^8  - 1657847338092 c y_1^{10}  \\
+& 29261950884 y_1^{12},\\
Q_{18}= \quad& 6145346701314 c^6  + 1705659405062004 c^5 y_1^2  - 177158182581522 c^4 y_1^4  \\
-& 87108572161368 c^3 y_1^6  + 14676006828366 c^2 y_1^8  - 656114701212 c y_1^{10}  \\
+& 5566963010 y_1^{12},\\
Q_{20}= \quad& 17073860098680 c^6  + 1765021213832760 c^5 y_1^2  - 262032349848480 c^4 y_1^4  \\
-& 25663081979952 c^3 y_1^6  + 4778135352216 c^2 y_1^8  - 147573429768 c y_1^{10}  \\
-& 751835760 y_1^{12},\\
Q_{22}= \quad& 9338536521864 c^6  + 941182952949648 c^5 y_1^2  - 130448854754568 c^4 y_1^4  \\
-& 5532982699104 c^3 y_1^6  + 1153495908792 c^2 y_1^8  - 23152890672 c y_1^{10}  \\
-& 611376440 y_1^{12},\\
Q_{24}= \quad& 205652969940096 c^5 y_1^2  - 16668899694720 c^4 y_1^4  - 2570855428224 c^3 y_1^6  \\
+& 271365697152 c^2 y_1^8  - 3749437440 c y_1^{10}  - 144665088 y_1^{12},\\
Q_{26}= \quad& 3762339840000 c^4 y_1^4  - 676026777600 c^3 y_1^6  + 35234611200 c^2 y_1^8  \\
-& 235929600 c y_1^{10}  - 15564800 y_1^{12}.
\end{aligned}
\end{equation*}
From (4.36), (4.37), eliminating $\kappa$ we get a polynomial of $y_1$ with constant coefficients of degree 428 and then $y_1$ is a constant, which is a contradiction.
\begin{description}
\item[Case B] $a = 0$. (4.4)-(4.6) becomes
\end{description}
\begin{equation}
\omega_{22}^{1} \omega_{33}^{1}=-\lambda_{2} \lambda_{3}-c,
\end{equation}
\begin{equation}
\omega_{22}^{1} \omega_{44}^{1}=-\lambda_{2} \lambda_{4}-c,
\end{equation}
\begin{equation}
\omega_{33}^{1} \omega_{44}^{1}=-\lambda_{3} \lambda_{4}-c.
\end{equation}
From the assumption $c \neq 0$ and (4.38)-(4.40), we can get $\omega_{ii}^{1} \neq 0$ and $\lambda_i \lambda_j+c \neq 0$ for $2 \leq i,j \leq 4, i \neq j$, then
\begin{equation}
\omega_{33}^{1} =-\frac{\lambda_{2} \lambda_{3}+c}{\omega_{22}^{1}},
\end{equation}
\begin{equation}
\omega_{44}^{1} =-\frac{\lambda_{2} \lambda_{4}+c}{\omega_{22}^{1}},
\end{equation}
\begin{equation}
(\omega_{22}^{1})^2 =-\frac{(\lambda_{2} \lambda_{3}+c)(\lambda_{2} \lambda_{4}+c)}{\lambda_{3} \lambda_{4}+c}.
\end{equation}
From (2.13), (2.14), (4.41)-(4.43), we have
\begin{equation}
e_1(y_1)=-\frac{1}{3\omega_{22}^{1}(\lambda_3 \lambda_4+c)}(6 c^2 y_1 + 5 c y_1 y_2 - 9 c y_3 + 4 y_1^2 y_3 - 6 y_2 y_3),
\end{equation}
\begin{equation}
e_1(y_2)=-\frac{1}{3\omega_{22}^{1}(\lambda_3 \lambda_4+c)}(2 c^2 y_1^2 + 6 c^2 y_2 + c y_1^2 y_2 + 6 c y_2^2 - 3 c y_1 y_3 + 5 y_1 y_2 y_3 - 9 y_3^2),
\end{equation}
\begin{equation}
e_1(y_3)=-\frac{1}{3\omega_{22}^{1}(\lambda_3 \lambda_4+c)}(c^2 y_1 y_2 + 9 c^2 y_3 + 2 c y_1^2 y_3 + 6 c y_2 y_3 + 6 y_1 y_3^2).
\end{equation}
Using (4.41)-(4.43), (4.15) can be rewritten as
\begin{equation}
\begin{aligned}
&27 c^4 y_1 - 7 c^3 y_1^3 + 36 c^3 y_1 y_2 - 7 c^2 y_1^3 y_2 + 7 c^2 y_1 y_2^2 + 27 c^3 y_3 \\
+& 27 c^2 y_1^2 y_3 - 7 c y_1^4 y_3 + 45 c^2 y_2 y_3 + c y_1^2 y_2 y_3 + 24 c y_2^2 y_3 + 54 c y_1 y_3^2 \\
-& 15 y_1^3 y_3^2 + 39 y_1 y_2 y_3^2 - 27 y_3^3 = 0.
\end{aligned}
\end{equation}
Differentiating (4.47) with respect to $e_1$, we have
\begin{equation}
\begin{aligned}
&162 c^6 y_1 - 54 c^5 y_1^3 - 14 c^4 y_1^5 + 594 c^5 y_1 y_2 - 182 c^4 y_1^3 y_2 - 14 c^3 y_1^5 y_2 \\
+& 567 c^4 y_1 y_2^2 - 132 c^3 y_1^3 y_2^2 + 143 c^3 y_1 y_2^3 + 900 c^4 y_1^2 y_3 - 238 c^3 y_1^4 y_3 - 14 c^2 y_1^6 y_3 \\
+& 351 c^4 y_2 y_3 + 1395 c^3 y_1^2 y_2 y_3 - 328 c^2 y_1^4 y_2 y_3 + 765 c^3 y_2^2 y_3 + 420 c^2 y_1^2 y_2^2 y_3 + 390 c^2 y_2^3 y_3 \\
+& 513 c^3 y_1 y_3^2 + 444 c^2 y_1^3 y_3^2 - 214 c y_1^5 y_3^2 + 1890 c^2 y_1 y_2 y_3^2 - 23 c y_1^3 y_2 y_3^2 + 1269 c y_1 y_2^2 y_3^2 \\
-& 1620 c^2 y_3^3 + 981 c y_1^2 y_3^3 - 360 y_1^4 y_3^3 - 1593 c y_2 y_3^3 + 1089 y_1^2 y_2 y_3^3 - 234 y_2^2 y_3^3 - 837 y_1 y_3^4 = 0.
\end{aligned}
\end{equation}
Differentiating (4.48) with respect to $e_1$, we have
\begin{equation}
\begin{aligned}
&972 c^8 y_1 + 216 c^7 y_1^3 - 784 c^6 y_1^5 - 28 c^5 y_1^7 \\
+& 7938 c^7 y_1 y_2 - 1416 c^6 y_1^3 y_2 - 1802 c^5 y_1^5 y_2 - 28 c^4 y_1^7 y_2 \\
+& 17091 c^6 y_1 y_2^2 - 4395 c^5 y_1^3 y_2^2 - 1026 c^4 y_1^5 y_2^2 + 13836 c^5 y_1 y_2^3 \\
-& 2715 c^4 y_1^3 y_2^3 + 3679 c^4 y_1 y_2^4 - 1458 c^7 y_3 + 19926 c^6 y_1^2 y_3 \\
-& 2736 c^5 y_1^4 y_3 - 2000 c^4 y_1^6 y_3 - 28 c^3 y_1^8 y_3 - 1053 c^6 y_2 y_3 \\
+& 65034 c^5 y_1^2 y_2 y_3 - 13479 c^4 y_1^4 y_2 y_3 - 2266 c^3 y_1^6 y_2 y_3 + 11610 c^5 y_2^2 y_3 \\
+& 67221 c^4 y_1^2 y_2^2 y_3 - 11346 c^3 y_1^4 y_2^2 y_3 + 19611 c^4 y_2^3 y_3 + 21341 c^3 y_1^2 y_2^3 y_3 \\
+& 8502 c^3 y_2^4 y_3 - 10287 c^5 y_1 y_3^2 + 40437 c^4 y_1^3 y_3^2 - 11912 c^3 y_1^5 y_3^2 \\
-& 1276 c^2 y_1^7 y_3^2 + 13716 c^4 y_1 y_2 y_3^2 + 72372 c^3 y_1^3 y_2 y_3^2 - 17465 c^2 y_1^5 y_2 y_3^2 \\
+& 64944 c^3 y_1 y_2^2 y_3^2 + 28074 c^2 y_1^3 y_2^2 y_3^2 + 39249 c^2 y_1 y_2^3 y_3^2 - 51516 c^4 y_3^3 \\
+& 3348 c^3 y_1^2 y_3^3 + 13011 c^2 y_1^4 y_3^3 - 9008 c y_1^6 y_3^3 - 115587 c^3 y_2 y_3^3 \\
+& 78183 c^2 y_1^2 y_2 y_3^3 - 304 c y_1^4 y_2 y_3^3 - 80649 c^2 y_2^2 y_3^3 + 68562 c y_1^2 y_2^2 y_3^3 \\
-& 14634 c y_2^3 y_3^3 - 94203 c^2 y_1 y_3^4 + 28710 c y_1^3 y_3^4 - 12240 y_1^5 y_3^4 \\
-& 113724 c y_1 y_2 y_3^4 + 42399 y_1^3 y_2 y_3^4 - 19620 y_1 y_2^2 y_3^4 + 21870 c y_3^5 \\
-& 33237 y_1^2 y_3^5 + 9234 y_2 y_3^5 = 0.
\end{aligned}
\end{equation}
From(4.47), (4.48), eliminating $y_2$ we have
\begin{equation}
\begin{aligned}
&(c y_1 - 3 y_3) (c^3 - c y_1 y_3 - 2 y_3^2) (961551 c^7 y_1^4 - 550638 c^6 y_1^6 \\
+& 92295 c^5 y_1^8 + 1078 c^4 y_1^{10} + 3241134 c^6 y_1^3 y_3 \\
-& 3718062 c^5 y_1^5 y_3 + 852336 c^4 y_1^7 y_3 - 16548 c^3 y_1^9 y_3 \\
+& 2071089 c^5 y_1^2 y_3^2 - 10454265 c^4 y_1^4 y_3^2 + 4078449 c^3 y_1^6 y_3^2 \\
-& 228795 c^2 y_1^8 y_3^2 - 4234032 c^4 y_1 y_3^3 - 9157698 c^3 y_1^3 y_3^3 \\
+& 9454374 c^2 y_1^5 y_3^3 - 621684 c y_1^7 y_3^3 - 4185918 c^3 y_3^4 \\
+& 4435965 c^2 y_1^2 y_3^4 + 10034766 c y_1^4 y_3^4 - 479115 y_1^6 y_3^4 \\
+& 7050888 c y_1 y_3^5 + 4106700 y_1^3 y_3^5 + 2217618 y_3^6) = 0.
\end{aligned}
\end{equation}
From(4.47), (4.49), eliminating $y_2$ we have
\begin{equation}
\begin{aligned}
&(c y_1 - 3 y_3) (c^3 - c y_1 y_3 - 2 y_3^2)(121234158 c^{12} y_1^5 + 4502127582 c^{11} y_1^7 \\
-& 4130694792 c^{10} y_1^9 + 1312681986 c^9 y_1^{11} - 138793970 c^8 y_1^{13} \\
-& 1131900 c^7 y_1^{15} - 6493539798 c^{11} y_1^4 y_3 + 17508226059 c^{10} y_1^6 y_3 \\
-& 26579271393 c^9 y_1^8 y_3 + 12168575541 c^8 y_1^{10} y_3 - 1901077395 c^7 y_1^{12} y_3 \\
+& 41618150 c^6 y_1^{14} y_3 - 19265759766 c^{10} y_1^3 y_3^2 + 4495598658 c^9 y_1^5 y_3^2 \\
-& 42049903410 c^8 y_1^7 y_3^2 + 42284538216 c^7 y_1^9 y_3^2 - 10550910060 c^6 y_1^{11} y_3^2 \\
+& 560212114 c^5 y_1^{13} y_3^2 - 6970891914 c^9 y_1^2 y_3^3 - 13811443002 c^8 y_1^4 y_3^3 \\
+& 83680475085 c^7 y_1^6 y_3^3 + 21298486113 c^6 y_1^8 y_3^3 - 21213196659 c^5 y_1^{10} y_3^3 \\
+& 1830863937 c^4 y_1^{12} y_3^3 + 30101526828 c^8 y_1 y_3^4 + 45869158224 c^7 y_1^3 y_3^4 \\
+& 279229797288 c^6 y_1^5 y_3^4 - 232982629074 c^5 y_1^7 y_3^4 + 23242316040 c^4 y_1^9 y_3^4 \\
-& 602724402 c^3 y_1^{11} y_3^4 + 23531498892 c^7 y_3^5 + 95793204717 c^6 y_1^2 y_3^5 \\
-& 35556908661 c^5 y_1^4 y_3^5 - 505869269463 c^4 y_1^6 y_3^5 + 173198646333 c^3 y_1^8 y_3^5 \\
-& 12359609514 c^2 y_1^{10} y_3^5 - 28067406876 c^5 y_1 y_3^6 - 471693791502 c^4 y_1^3 y_3^6 \\
-& 188063792496 c^3 y_1^5 y_3^6 + 305129824146 c^2 y_1^7 y_3^6 - 20136670848 c y_1^9 y_3^6 \\
-& 99047493522 c^4 y_3^7 - 128261179269 c^3 y_1^2 y_3^7 + 309915643131 c^2 y_1^4 y_3^7 \\
+& 254417953509 c y_1^6 y_3^7 - 9828590625 y_1^8 y_3^7 + 92615758344 c^2 y_1 y_3^8 \\
+& 252963396576 c y_1^3 y_3^8 + 90957709080 y_1^5 y_3^8 \\
+& 45154416006 c y_3^9 + 25402919166 y_1^2 y_3^9) = 0.
\end{aligned}
\end{equation}
Next we check three subcases.
\begin{description}
\item[Case B.1] $c y_1 - 3 y_3 = 0$. Substituting $y_3=\dfrac{c y_1}{3}$ into (4.47), (4.48) gives
\end{description}
\begin{equation}
(9 c + 4 y_1^2 + 9 y_2) (12 c - 3 y_1^2 + 5 y_2)  = 0,
\end{equation}
\begin{equation}
(9 c + 4 y_1^2 + 9 y_2) (162 c^2 + 171 c y_1^2 - 94 y_1^4 + 549 c y_2 + 19 y_1^2 y_2 + 273 y_2^2) = 0.
\end{equation}
From (4.52), (4.53), we consider the two following subcases: \\
(i) $9 c + 4 y_1^2 + 9 y_2 \neq 0$, or \\
(ii) $9 c + 4 y_1^2 + 9 y_2=0$.\\
If (i) holds, by eliminating $y_2$ we can get that $y_1$ satisfies a polynomial with constant coefficients and $y_1$ is a constant. If (ii) holds, taking $9 c + 4 y_1^2 + 9 y_2=0$ into (4.44) we get $e_1(y_1)=0$ and $y_1$ is also a constant. So in both subcases we get that $y_1$ is a constant, which is a contradiction.
\begin{description}
\item[Case B.2] $c^3-c y_1 y_3-2 y_3^2 = 0$. Since $c \neq 0$, we solve for $y_1$ in terms of $y_3$.
\end{description}
Substituting $y_1=\dfrac{c^3 -2y_3^2}{c y_3}$ into (4.47) and (4.48), we have
\begin{equation}
(2 c^3 + c^2 y_2 - y_3^2) (7 c^9 - 65 c^6 y_3^2 - 7 c^5 y_2 y_3^2 + 89 c^3 y_3^4 - 10 c^2 y_2 y_3^4 + 8 y_3^6) = 0,
\end{equation}
\begin{equation}
\begin{aligned}
&(2 c^3 + c^2 y_2 - y_3^2) (14 c^{15} + 106 c^{12} y_3^2 + 132 c^{11} y_2 y_3^2 - 2004 c^9 y_3^4 \\
-& 1493 c^8 y_2 y_3^4 - 143 c^7 y_2^2 y_3^4 + 5885 c^6 y_3^6 + 2429 c^5 y_2 y_3^6 \\
-& 104 c^4 y_2^2 y_3^6 - 3842 c^3 y_3^8 - 68 c^2 y_2 y_3^8 + 192 y_3^{10}) = 0.
\end{aligned}
\end{equation}
From (4.54), (4.55), we consider the two following subcases: \\
(i) $2c^3+c^2 y_2-y_3^2 \neq 0$, or \\
(ii) $2c^3+c^2 y_2-y_3^2=0$.\\
If (i) holds, by eliminating $y_2$ we can get that $y_3$ satisfies a polynomial with constant coefficients and $y_3$ is a constant. If (ii) holds, taking $2c^3+c^2 y_2-y_3^2=0$ into (4.46) we get $e_1(y_3)=0$ and $y_3$ is also a constant. So in both subcases we get that $y_3$ is a constant, then $y_1$ is a constant, which is a contradiction.
\begin{description}
\item[Case B.3] $(c y_1 - 3 y_3)(c^3-c y_1 y_3-2 y_3^2) \neq 0$. From (4.50), (4.51), we have
\end{description}
\begin{equation}
\begin{aligned}
&961551 c^7 y_1^4 - 550638 c^6 y_1^6 + 92295 c^5 y_1^8 + 1078 c^4 y_1^{10} \\
+& 3241134 c^6 y_1^3 y_3 - 3718062 c^5 y_1^5 y_3 + 852336 c^4 y_1^7 y_3 - 16548 c^3 y_1^9 y_3 \\
+& 2071089 c^5 y_1^2 y_3^2 - 10454265 c^4 y_1^4 y_3^2 + 4078449 c^3 y_1^6 y_3^2 - 228795 c^2 y_1^8 y_3^2 \\
-& 4234032 c^4 y_1 y_3^3 - 9157698 c^3 y_1^3 y_3^3 + 9454374 c^2 y_1^5 y_3^3 - 621684 c y_1^7 y_3^3 \\
-& 4185918 c^3 y_3^4 + 4435965 c^2 y_1^2 y_3^4 + 10034766 c y_1^4 y_3^4 - 479115 y_1^6 y_3^4 \\
+& 7050888 c y_1 y_3^5 + 4106700 y_1^3 y_3^5 + 2217618 y_3^6 = 0,
\end{aligned}
\end{equation}
\begin{equation}
\begin{aligned}
&121234158 c^{12} y_1^5 + 4502127582 c^{11} y_1^7 - 4130694792 c^{10} y_1^9 \\
+& 1312681986 c^9 y_1^{11} - 138793970 c^8 y_1^{13} - 1131900 c^7 y_1^{15} \\
-& 6493539798 c^{11} y_1^4 y_3 + 17508226059 c^{10} y_1^6 y_3 - 26579271393 c^9 y_1^8 y_3 \\
+& 12168575541 c^8 y_1^{10} y_3 - 1901077395 c^7 y_1^{12} y_3 + 41618150 c^6 y_1^{14} y_3 \\
-& 19265759766 c^{10} y_1^3 y_3^2 + 4495598658 c^9 y_1^5 y_3^2 - 42049903410 c^8 y_1^7 y_3^2 \\
+& 42284538216 c^7 y_1^9 y_3^2 - 10550910060 c^6 y_1^{11} y_3^2 + 560212114 c^5 y_1^{13} y_3^2 \\
-& 6970891914 c^9 y_1^2 y_3^3 - 13811443002 c^8 y_1^4 y_3^3 + 83680475085 c^7 y_1^6 y_3^3 \\
+& 21298486113 c^6 y_1^8 y_3^3 - 21213196659 c^5 y_1^{10} y_3^3 + 1830863937 c^4 y_1^{12} y_3^3 \\
+& 30101526828 c^8 y_1 y_3^4 + 45869158224 c^7 y_1^3 y_3^4 + 279229797288 c^6 y_1^5 y_3^4 \\
-& 232982629074 c^5 y_1^7 y_3^4 + 23242316040 c^4 y_1^9 y_3^4 - 602724402 c^3 y_1^{11} y_3^4 \\
+& 23531498892 c^7 y_3^5 + 95793204717 c^6 y_1^2 y_3^5 - 35556908661 c^5 y_1^4 y_3^5 \\
-& 505869269463 c^4 y_1^6 y_3^5 + 173198646333 c^3 y_1^8 y_3^5 - 12359609514 c^2 y_1^{10} y_3^5 \\
\end{aligned}
\end{equation}
\newpage
\begin{equation*}
\begin{aligned}
-& 28067406876 c^5 y_1 y_3^6 - 471693791502 c^4 y_1^3 y_3^6 - 188063792496 c^3 y_1^5 y_3^6 \\
+& 305129824146 c^2 y_1^7 y_3^6 - 20136670848 c y_1^9 y_3^6 - 99047493522 c^4 y_3^7 \\
-& 128261179269 c^3 y_1^2 y_3^7 + 309915643131 c^2 y_1^4 y_3^7 + 254417953509 c y_1^6 y_3^7 \\
-& 9828590625 y_1^8 y_3^7 + 92615758344 c^2 y_1 y_3^8 + 252963396576 c y_1^3 y_3^8 \\
+& 90957709080 y_1^5 y_3^8 + 45154416006 c y_3^9 + 25402919166 y_1^2 y_3^9 = 0.
\end{aligned}
\end{equation*}
From (4.56) and (4.57), eliminating $y_3$ we get a polynomial of $y_1$ with constant coefficients of degree 118 and then $y_1$ is a constant, that is, $\lambda_1$ is constant, which is a contradiction. Therefore we complete the proof of Theorem 1.1.

 \par
 \vskip 6pt
 \footnotesize{}

\end{document}